\documentclass[11pt,a4paper,reqno]{amsart}
\usepackage{amsmath, amsthm, amsfonts, amssymb, mathrsfs}
\usepackage{graphicx,xcolor}
\usepackage{mathtools}
\usepackage{enumerate}
\usepackage[update,prepend]{epstopdf}
\usepackage{color}
\usepackage{soul}

\usepackage{marginnote}
\usepackage{hyperref}
\newcommand{\R}{{\mathbb R}} 

\newtheorem{theorem}{Theorem}
\newtheorem{corollary}[theorem]{Corollary}
\newtheorem{lemma}{Lemma}

\theoremstyle{plain}
\begin{document}
\title[Functional-differential equations with Stolarsky means]
{On functions whose mean value abscissas are Stolarsky means of endpoints}
\author{Zarif O.~Ibragimov\;}
\address{Mathematisches Institut, Universit\"at Bern\\
	Sidlerstr. 5\\
	CH-3012 Bern, Switzerland}
\email{zarif.ibragimov@math.unibe.ch}
\author{\;Farrukh Mukhammadiev}
\address{Facult\'e des HEC, Universit\'e de Lausanne\\
	Internef, Quartier de Chamberonne\\
	CH-1015 Lausanne, Switzerland}
\email{farrukh.mukhammadiev@gmail.com}

\begin{abstract}
We provide a new approach to determine all differentiable 
functions whose mean value abscissas are Stolarsky means of 
endpoints.
\end{abstract}
\maketitle

\section{Introduction and the main result}

\bigskip

Let $I\subset\R$ be an open interval and let $f\colon I\to\R$ be a differentiable function. For any finite interval $(a,b)\subset I$, 
Lagrange's Mean Value Theorem guarantees the existence of $c\in (a,b)$ such that the tangent to the graph of $f$ at the point $(c, f(c))$ 
is parallel to the secant through its endpoints $(a, f(a))$ and $(b, f(b))$.

It is one of the many beautiful properties of the parabolas with horizontal directrix that the mean value in the Lagrange Mean Value Theorem always corresponds 
to the arithmetic mean of the endpoints. That is, we have
\begin{equation}\label{prop.parabola}
f(b)-f(a)=(b-a)f'\Bigl(\frac{a+b}{2}\Bigr) 
\quad \text{for all} \quad a,b\in\R,
\end{equation}
if $f$ is any quadratic polynomial. It is then a natural question to ask whether quadratic polynomials are the only members of the 
vector space of differentiable functions having the property~\eqref{prop.parabola}. More generally, by asking for which $f$ the mean 
value $c$ in Lagrange's Mean Value Theorem depends on the endpoints in a certain given manner, one arrives at so-called functional-differential equations.

It was Haruki~\cite{Haruki} and Acz\'el~\cite{Aczel} who showed independently that the quadratic polynomials are the only differentiable functions that 
solve~\eqref{prop.parabola}. The functional-differential equation \eqref{prop.parabola} was one of the starting points for the rich literature 
devoted to various and more general functional equations including the ones for higher order Taylor expansions and also in the abstract setting of groups. 
To name just a few, we mention \cite{Sablik}, \cite{Kannappan}, \cite{Ebanks}, \cite{Pales}, \cite{BIM} and also \cite{Sahoo-Riedel} for a textbook 
reference.

In this note we are interested in the case when the mean value in Lagrange's Mean Value Theorem corresponds to the so-called \emph{Stolarsky mean} of end-points. 
More precisely, we solve the functional-differential equation
\begin{equation}\label{FDE.stolyar}
f(b)-f(a)=(b-a)f'(S_\alpha(a,b)), \quad a,b>0,
\end{equation}
where, for $\alpha\in(-\infty,\infty)$, 
$S_{\alpha}:(0,\infty)^2\to (0,\infty)$ stands for the 
Stolarsky mean which is a continuous function defined 
by
\begin{equation*}\label{Stolarsky.mean}
S_\alpha(a,b)=
	\begin{cases}
	\displaystyle
		\Bigl(\frac{b^{\alpha}-a^{\alpha}}{\alpha(b-a)} 
		\Bigr)^{\frac{1}{\alpha-1}} & \text{if } 
	\alpha\notin\{0, 1\},\\[1.5ex]
	S_0(a,b) & \text{if } \alpha=0,\\[1ex]
	S_1(a,b) & \text{if } \alpha=1,\\
	\end{cases}
\end{equation*}
for all $a, b>0$ with $a\neq b$, and $S_\alpha(a,a):=a$ 
for all $a>0$, where
\begin{equation*}\label{S0andS1}
S_0(a,b):=
\frac{b-a}{\log b-\log a}, \quad
S_1(a,b):=
\frac{1}{e}\exp \bigg(\frac{b\log b-a \log a}{b-a}\bigg),
\end{equation*}
see e.g.~\cite{Stolarsky}. 
For example, the case $\alpha=2$ corresponds to the 
arithmetic mean. 

Our main result reads as follows.

\begin{theorem}\label{thm:main}
Let $\alpha\in\R$ be arbitrary and let $f\colon(0,\infty)\to\R$ be 
a differentiable function satisfying~\eqref{FDE.stolyar}.
\begin{enumerate}[\upshape (i)]
	\item If $\alpha\notin\{0,1\}$, then there exist constants 
	$c_1, c_2, c_3\in\R$ such that
	\begin{equation}\label{sol.case.1}
		f(x)=c_1x^\alpha+c_2x+c_3, \quad x>0.
	\end{equation}
	\item If $\alpha\in\{0,1\}$, then there exist constants 
	$c_1, c_2, c_3\in\R$ such that
	\begin{equation}\label{sol.case.2}
	f(x)=c_1x^\alpha\log(x)+c_2x+c_3, \quad x>0.
	\end{equation}
\end{enumerate}
\end{theorem}

In particular, we have the following result for the case 
when the mean value in Lagrange's Mean Value Theorem 
corresponds to the geometric mean of endpoints.

\begin{corollary}
	Let $f\colon(0,\infty)\to\R$ be a differentiable function such that
	\begin{equation}\label{prop.hyperbola}
	f(b)-f(a)=(b-a)f'\bigl(\sqrt{ab}\bigr) 
	\quad \text{for all} \quad a,b>0.
	\end{equation}
	Then there exist constants 
	$c_1, c_2, c_3\in\R$ such that
	\begin{equation}\label{gm.sol.}
	f(x)=\frac{c_1}{x}+c_2x+c_3, \quad x>0.
	\end{equation}
\end{corollary}

It is easy to check that all the functions of the form 
\eqref{sol.case.1}-\eqref{sol.case.2} indeed satisfy~\eqref{FDE.stolyar}. 
To the best of our knowledge, the first and the only work in 
the current literature which considers \eqref{FDE.stolyar} 
is the paper~\cite{Matko99}. 
The purpose of this note is to provide a direct approach based on reducing 
the problem to elementary differential equations with a novel trick 
inspired by the methods of~\cite{BIM}.

\section{Proof of Theorem~\ref{thm:main}}
First we establish a preliminary result on the smoothness 
of the solutions of~\eqref{FDE.stolyar}.

\begin{lemma}\label{lem:key}
Let $\alpha\in\R$ be arbitrary. Every differentiable function 
$f:(0, \infty)\to\R$ satisfying \eqref{FDE.stolyar} is 
infinitely differentiable on $(0, \infty)$.
\end{lemma}
\begin{proof}
We provide a detailed proof for the case $\alpha\notin\{0,1\}$.
The analyses of the cases $\alpha=0$ and $\alpha=1$ are 
completely similar and are left to the interested reader.
Let us take arbitrary $x_0>0$. Below we show infinite 
differentiability of $f$ in a neighborhood of $x_0$.
It is easy to see that there are $y_0>0$ and $h_0>0$ 
such that 
$(y_0+h_0)^{\alpha}-y_0^{\alpha}=\alpha h_0 x_0^{\alpha-1}$.
For example, we can take 
\begin{equation*}
y_0=h_0=\Bigl(\frac{2^\alpha-1}{\alpha}\Bigr)^{\frac{1}{1-\alpha}}x_0.
\end{equation*}
Next, we consider the function 
\begin{equation*}
\psi(x,y)=(y+h_0)^{\alpha}-y^{\alpha}-\alpha h_0 x^{\alpha-1}.
\end{equation*}
We have $\psi(x_0,y_0)=0$. Clearly, $\psi$ is continuously 
differentiable with 
\[
\frac{\partial}{\partial y}\psi(x,y)
	=\alpha(y+h_0)^{\alpha-1}-\alpha y^{\alpha-1}
\]
which does not vanish since $\alpha\neq1$ and $\alpha\neq0$. 
By the implicit 
function theorem there is a neighborhood $U$ of $x_0$ and 
a unique continuously differentiable function 
$\phi:U\to(0,\infty)$ such that $\phi(x_0)=y_0$ and 
$\psi(x,\phi(x))=0$ for all $x\in U$.
Furthermore,

\begin{equation}\label{PD.phi.x}
\phi'(x)=\frac{(\alpha-1)h_0x^{\alpha-2}}{(\phi(x)+h_0)^{\alpha-1}-\phi(x)^{\alpha-1}},
\quad x\in U.
\end{equation}
Inserting $\phi$ and $\phi+h_0$ in \eqref{FDE.stolyar} 
for $a$ and $b$, respectively, we get
\begin{equation}\label{eqn.inserted}
f(h_0+\phi(x))-f(\phi(x))=h_0f'(x), \quad x\in U.
\end{equation}
It follows from \eqref{PD.phi.x} that $\phi$ is infinitely 
differentiable in $U$. 
Hence, if we assume that $f$ is $k\ge 1$ times differentiable, 
then \eqref{eqn.inserted} implies its $k+1$ times differentiability 
in $U$. By induction, it follows that $f$ is infinitely 
differentiable in the neighborhood $U$ of $x_0$. 
\end{proof}

\noindent
\textit{Proof of Theorem~\emph{\ref{thm:main}}}.
First let $\alpha\notin\{0,1\}$.
Let $t>0$ be fixed arbitrarily.
For every $r>0$, there is $x_0>0$ such that 
$(t+r)^{\alpha}-t^{\alpha}=\alpha r x_0^{\alpha-1}$.
Next, we consider the function 
\begin{equation}\label{Psi}
\Psi(h,y)=(y+h)^{\alpha}-y^{\alpha}-\alpha h x_0^{\alpha-1}
\end{equation}
on $(0,\infty)^2$. We have $\Psi(r,t)=0$. It is easy to check that the implicit 
function theorem applies and yields the existence of a 
neighborhood $V$ of $r$, and a unique continuously 
differentiable function $g\colon V\to(0,\infty)$ such 
that $g(r)=t$ and $\Psi(h,g(h))=0$ for all $h\in V$. 
Moreover, we have
\begin{equation}\label{PD.g.h}
g'(h)
	=\frac{1}{\alpha h}\frac{(g(h)+h)^\alpha-g(h)^\alpha-\alpha h(g(h)+h)^{\alpha-1}}{(g(h)+h)^{\alpha-1}-g(h)^{\alpha-1}}.
\end{equation}
By inserting $g$ and $h+g$ in \eqref{FDE.stolyar} 
for $a$ and $b$, respectively, and taking into account the 
fact that $\Psi(h,g(h))=0$ for all $h\in V$, we get
\begin{equation}\label{eqn.inserted.2}
f(h+g(h))-f(g(h))=hf'(x_0), \quad h\in V.
\end{equation}
Next, let us differentiate \eqref{eqn.inserted.2} 
with respect to $h$ to get
\begin{equation}\label{eqn.diff.1}
f'(h+g(h))+f'(h+g(h)g'(h)-f'(g(h))g'(h)=f'(x_0).
\end{equation}
If we multiply \eqref{eqn.diff.1} by $\frac{1}{h^2}$ and 
differentiate the resulting identity with respect to $h$, 
and insert \eqref{PD.g.h}, then for $h=r$, we obtain, 
\begin{equation}\label{eqn.func.1}
\frac{1}{\alpha r}\Bigl(f'(r+t)-f'(t)\Bigr) = 
	\frac{R(r,t)}{S(r,t)}+(f'(x_0)-f'(t))\frac{T(r,t)}{S(r,t)},
\end{equation}
for all $r>0$, where
\begin{equation*}
\begin{aligned}
R(r,t):=&\frac{f''(r+t)}{r^3}\phi(r,t)^2
		-\frac{f''(t)}{r^3}\psi(r,t)^2\\[2ex]
S(r,t):=&\frac{\alpha}{r^3}\bigl((r+t)^{\alpha-1}-t^{\alpha-1}\bigr)\phi(r,t)
			-\frac{\alpha(\alpha-1)}{r^2}t^{\alpha-2}\psi(r,t)\\[2ex]
T(r,t):=&\frac{\alpha-1}{r^3}\Big((r+t)^{\alpha-2}\phi(r,t)-t^{\alpha-2}\psi(r,t)\Big)
			-\frac{1}{r^4}\bigl((r+t)^{\alpha-1}-t^{\alpha-1}\bigr)
\end{aligned}
\end{equation*}
with
\begin{equation*}
\phi(r,t):=\frac{(r+t)^{\alpha}-t^{\alpha}}{\alpha r}-t^{\alpha-1}, 
\quad 	
\psi(r,t):=\frac{(r+t)^{\alpha}-t^{\alpha}}{\alpha r}-(r+t)^{\alpha-1}.
\end{equation*}
It is not difficult to check that the following asymptotic 
expansions hold 
%
\begin{equation*}
\begin{aligned}
R(r,t)&=\frac{(\alpha-1)^2}{4}t^{2\alpha-4}f'''(t)
-\frac{(\alpha-1)^2(\alpha-2)}{6}t^{2\alpha-5}f''(t) + {\text{o}}(1),\\
S(r,t)&=\frac{\alpha(\alpha-1)^2(\alpha-2)}{12}t^{2\alpha-5} + {\text{o}}(1),\\
T(r,t)&=-\frac{(\alpha-1)^2(\alpha-2)t^{2\alpha-6}}{12} + {\text{o}}(1),
\end{aligned}
\end{equation*}
as $r\searrow0$. Furthermore, from the equation $\Psi(r,t)=0$, we get
\begin{equation*}
x_0=\bigg(\frac{(t+r)^{\alpha}-t^\alpha}{\alpha r}\bigg)^{\frac{1}{\alpha-1}}=t+{\text{o}}(1), \quad \text{as} \quad r\searrow0.
\end{equation*}
In view of these asymptotic expansions, letting $r\searrow0$ 
in \eqref{eqn.func.1}, we deduce the elementary differential 
equation
\begin{equation}\label{simple.ODE}
f'''(t)-\frac{\alpha-2}{t}f''(t)=0.
\end{equation}
Since $t>0$ was fixed arbitrarily, \eqref{simple.ODE} must 
hold for all $t>0$. Taking into account the hypothesis that
$\alpha(\alpha-1)\neq 0$, it follows immediately that 
the solution space of the differential equation in \eqref{simple.ODE} 
is given by $\text{span}\{1,t,t^{\alpha}\}$, completing the proof 
of part (i). 

For the cases $\alpha=0$ and $\alpha=1$, in the  
same way as above one obtains the differential equation 
\eqref{simple.ODE} (for the corresponding values of $\alpha$), 
the solution spaces of which are $\text{span}\{1,t,\log t\}$ 
and $\text{span}\{1,t,t\log t\}$, respectively.
\qed

\medskip
\subsection*{Acknowledgements} 
The authors are grateful to Diyora Salimova for reading a preliminary version of this note and 
making some important comments.

%
\end{document}